\documentclass[12pt]{article}
\usepackage{graphicx}
\usepackage{titling}
\usepackage{fancyhdr}
\usepackage{amsthm}
\usepackage{lipsum}
\usepackage{amsfonts, amsmath}
\usepackage{hyperref}
\usepackage[margin=1.25in]{geometry}
\usepackage{tikz}
\usepackage{setspace}

\usepackage{titlesec}

\titleformat{\section}
  {\normalfont\centering\large}
  {\thesection}{1em}{}

\titleformat{\subsection}
  {\normalfont\normalsize\itshape}
  {\thesubsection}{1em}{}

\usepackage[
    backend=biber,
    style=alphabetic
]{biblatex}
\renewbibmacro{in:}{}
\usepackage{float}
\usepackage[normalem]{ulem}
\usepackage{verbatim}

\newtheorem{thm}{Theorem}[section]

\newtheorem{coro}[thm]{Corollary}
\newtheorem{lem}[thm]{Lemma}

\newtheorem{remk}[thm]{Remark}

\newcommand{\Hbb}{ {\mathbb H}}

\title{Hot spots in convex hyperbolic planar domains with small eigenvalues}
\author{Lawford Hatcher}
\date{}

\begin{document}

\maketitle

\begin{abstract}
    We prove a variant of Rauch's hot spots conjecture for convex hyperbolic planar domains with small Neumann or mixed Dirichlet-Neumann eigenvalues. We conclude, for instance, that on bounded convex domains in the hyperbolic plane with sufficiently large area, second Neumann Laplace eigenfunctions have no interior critical points. 
\end{abstract}

\section{Introduction}\label{intro}

Let $\Hbb^2$ denote the hyperbolic plane. Given a domain $\Omega\subseteq \Hbb^2$, let $-\Delta$ denote the Neumann Laplacian acting on $L^2(\Omega)$. 

\begin{thm}\label{neumann14}
    Suppose that $\Omega\subseteq \Hbb^2$ is a finite-area, geodesically convex domain. If $\mu\in (0,1/4]$ is the least element of the positive spectrum of $-\Delta$, then any Neumann eigenfunction corresponding to $\mu$ has no interior critical points.
\end{thm}
\indent The threshold $1/4$ arises naturally as the bottom of the spectrum of the Laplacian acting on the hyperbolic plane. Below this threshold, there exist smooth positive solutions to $-\Delta u=\mu u$ on $\Hbb^2$ that are globally radially decreasing. We use these functions to construct test functions that allow us to prove Theorem \ref{neumann14} via contradiction.\\
\indent Theorem \ref{neumann14} is not restricted entirely to bounded domains. In particular, we explain in Section \ref{exsection} below that an ideal $n$-gon $\Omega$ in the hyperbolic plane with at least $18$ vertices has a positive Neumann eigenvalue under $1/4$. Theorem \ref{neumann14} implies that an eigenfunction corresponding to the smallest such eigenvalue has critical points only in $\partial\Omega$.\\
\indent Another application of Theorem \ref{neumann14} is that large but bounded convex hyperbolic planar domains have no interior critical points. We prove the following in \S \ref{largesubsection} below:
\begin{coro}\label{largedomains}
    There exists $A>0$ such that if $\Omega\subseteq \Hbb^2$ is a bounded convex domain with area at least $A$, then a second Neumann eigenfunction on $\Omega$ has no interior critical points. The constant $A$ may be taken to be the unique value for which the hyperbolic disk of area $A$ has second Neumann eigenvalue equal to $1/4$.
\end{coro}
\begin{remk}\label{numremk}
    Based on the second statement in Corollary \ref{largedomains}, numerics indicate that we may take $A=33.35$. See Section \ref{exsection} below for details.
\end{remk}
\indent We also prove a variant of Theorem \ref{neumann14} for the mixed problem. Let $D\subseteq \partial\Omega$ be a connected subarc of the boundary. Let $-\Delta_M$ denote the mixed Laplacian on $\Omega$ with Dirichlet boundary conditions on $D$ and Neumann boundary conditions on $\partial\Omega\setminus D$.

\begin{thm}\label{mixed14}
    Suppose that $\Omega\subseteq \Hbb^2$ is geodesically convex with finite area. If $\lambda\in(0,1/4]$ is the least element of the spectrum of $-\Delta_M$, then any mixed eigenfunction corresponding to $\lambda$ has no interior critical points. 
\end{thm}

\noindent In particular, as we show in Section \ref{exsection} below, for any bounded convex domain $\Omega\subseteq\Hbb^2$, there exists $\epsilon>0$ depending on $\Omega$ such that if $D$ is connected and has diameter at most $\epsilon$, then a corresponding first mixed eigenfunction has no interior critical points.\\
\indent Rauch conjectured in 1974 that a second Neumann eigenfunction on a bounded Lipschitz domain in Euclidean space has extrema only on the boundary. This is widely known as the \textit{hot spots conjecture}. See \cite{rauch} and \cite{banuelosburdzy} for a discussion of this problem. There have been many recent advancements made on this problem; for instance \cite{jm, newapproach, pont, yaoevenefcns, yao1}.\\
\indent We recently addressed the hyperbolic version of this conjecture for non-acute geodesic triangles in $\Hbb^2$ \cite{me4}. However, as we point out in \cite{me4}, hyperbolic triangles always have second Neumann eigenvalue exceeding $1/4$, so Theorem \ref{neumann14} does not imply the result in this setting.  \\
\indent To our knowledge, Miyamoto \cite{miyamoto1} was the first to apply the method of proof we employ here to partially address the hot spots conjecture. We also used this method to prove the Euclidean version of Theorem \ref{mixed14} in \cite{me3}. Moreover, Rohleder \cite{rohledermiya} recently made use of this sort of argument to prove a result related to the hot spots conjecture.

\section{Proof of main theorems}

The proofs of Theorems \ref{neumann14} and \ref{mixed14} are very similar to each other. Recall that the spectrum of the Laplacian $-\Delta_{\Hbb^2}$ acting on $L^2(\Hbb^2)$ equals the interval $[1/4,\infty)$ and consists entirely of essential spectrum. In particular, for any function $u$ in the Sobolev space $H^1(\Hbb^2)$, we have 
\begin{equation}\label{14eqn}
    \frac{\displaystyle\int_{\Hbb^2}|\nabla u|^2}{\displaystyle\int_{\Hbb^2}|u|^2}>\frac14,
\end{equation} where the norm of the gradient and the integrals are taken with respect to the Riemannian metric on $\Hbb^2$. We refer the reader to \cite{terras} for proofs of these well-known facts.\\
\indent Given any $\mu>0$ and a point $p\in \Hbb^2$, there exists a smooth real-valued function $J=J_{p,\mu}$ on $\Hbb^2$ that satisfies $-\Delta J=\mu J$ pointwise and that depends only on the geodesic distance to the point $p$. This function can be described explicitly in terms of the Legendre function of degree $s$, where $s$ is either solution to $-\mu=s(s-1)$:
\begin{equation*}
    J_{p,\mu}(x)=P_{-s}(\cosh(r)),
\end{equation*}
 where $r$ is the distance from $x$ to $p$. See Exercise 3.2.11 \cite{terras}. Note that $J(p)=1$.

\begin{lem}\label{legendre}
    If $\mu\in (0,1/4]$, then the function $J$ is positive everywhere, decreasing as a function of $r$, and has a critical point at $p$.
\end{lem}
\begin{proof}
    That $J$ has a critical point at $p$ follows from being smooth and radial.\\
    \indent Suppose that $J$ vanishes at some $r=r_0$. Let $u$ equal the function equal to $J$ on the geodesic disk of radius $r_0$ about $p$ and equal to $0$ elsewhere. Then $u\in H^1(\Hbb^2)$, and, by integration by parts, \[\frac{\displaystyle\int_{\Hbb^2}|\nabla u|^2}{\displaystyle\int_{\Hbb^2}|u|^2}=\mu\leq \frac14,\] contradicting (\ref{14eqn}). Thus, $J>0$ on $\Hbb^2$. By the maximum principle, $J$ therefore has no local minima. It follows that $J$ must decrease as a function of $r$.
\end{proof}

The proof of Theorem \ref{neumann14} primarily consists of showing that a second Neumann eigenfunction cannot have an interior critical point at which the function is nonzero. That the eigenfunction has no nodal critical points is well known, but we provide a proof in Lemma \ref{nodcrit} below. Recall that, by the work of Cheng \cite{cheng}, the nodal (i.e. zero-level) set of an eigenfunction is a union of smooth arcs that intersect each other transversely, and these intersections occur exactly at the set of nodal critical points (i.e. points at which $u=0$ and $\nabla u=0$ simultaneously) of the eigenfunction. 

\begin{lem}\label{nodcrit}
    If $u$ is a Neumann eigenfunction on a simply connected domain $\Omega\subseteq \Hbb^2$ whose eigenvalue $\mu\in (0,1/4]$ is the least element of the positive spectrum, then $u$ has no nodal critical points. 
\end{lem}
\begin{proof}
    We claim that the nodal set of $u$ does not contain a loop. Indeed, suppose to the contrary that $\Omega'$ is a subdomain of $\Omega$ bounded by a loop in the nodal set. Then the function equal to $u$ in $\Omega'$ and equal to $0$ in $\Hbb^2\setminus \Omega'$ defines a test function that contradicts (\ref{14eqn}), proving the claim.\\
    \indent If $u$ has a nodal critical point, then near this point, the nodal set of $u$ consists of at least two transversely intersecting arcs. By the claim, these arcs cannot form a loop and thus have at least four distinct endpoints in $\partial\Omega$. Since $\Omega$ is simply connected, $u$ has at least four nodal domains, contradicting Courant's nodal domain theorem (see, e.g., \cite[Appendix D, Theorem 4 and Remark 6(3)]{BerardMeyer1982}). 
\end{proof}

\begin{proof}[Proof of Theorem \ref{neumann14}]
    Suppose toward a contradiction that $u$ is an eigenfunction as in the theorem statement with an interior critical point $p\in \Omega$. By Lemma \ref{nodcrit}, we may multiply $u$ by $-1$ if necessary to assume that $u(p)>0$. Let $J=J_{p,\mu}$ be as above, and define a function \[w(x)=u(p)J(x)-u(x).\] Then $w$ satisfies the equation $-\Delta w=\mu w$ pointwise in $\Omega$. Since $J(p)=1$ and has a critical point at $p$, the function $w$ has a nodal critical point at $p$. Since $\Omega$ has finite area and the function $J$ and its gradient are bounded, $w\in H^1(\Omega)$.\\
    \indent Near $p$, the nodal set of $w$ is therefore the union of at least two transversely intersecting arcs. By the same argument used in the proof of Lemma \ref{nodcrit}, the nodal set of $w$ does not contain a loop. Thus, the nodal arcs emanating from $p$ have at least four distinct endpoints in $\partial\Omega$ and partition $\Omega$ into at least four subdomains. On at least two of these subdomains $\Omega_1$ and $\Omega_2$, $w$ is positive. Let $u_i$ be the function equal to $w$ on $\Omega_i$ and equal to $0$ on $\Hbb^2\setminus \Omega_i$.\\
    \indent Let $q\in \partial\Omega_i\cap \partial\Omega$. Let $\partial_{\nu}$ denote the outward normal derivative. At $q$, by the Neumann boundary condition, we have \[\partial_{\nu}w=u(p)\partial_{\nu}J.\] Since $\Omega$ is convex and $J$ is decreasing as a function of $r$, we have \[\partial_{\nu}w(q)<0.\] Integration by parts then gives, for $i\in \{1,2\}$, \[\int_{\Omega}|\nabla u_i|^2=\mu\int_{\Omega}|u_i|^2+\int_{\partial\Omega_i\cap \partial\Omega}u_i\partial_{\nu}u_i\leq \mu\int_{\Omega}|u_i|^2.\] Since the $u_i$ have disjoint supports, they are $L^2(\Omega)$-orthogonal. The span of $u_1$ and $u_2$ provides a two dimensional space of functions on which the Rayleigh quotient is at most $\mu$. Since $\mu$ is the second Neumann eigenvalue, there exists a second Neumann eigenfunction in their span. However, any function in this span vanishes on an open set, contradicting unique continuation (see, e.g., \cite[Remark 4.3.19]{LevitinMangoubiPolterovich2023}).
\end{proof}

\begin{proof}[Proof of Theorem \ref{mixed14}]
    The proof is almost identical to Theorem \ref{neumann14}. Note in this case that the eigenfunction can be taken strictly positive in $\Omega$. We again work by contradiction and construct a new function $w$ as above and use this to partition $\Omega$ into nodal domains of $w$. Since the eigenfunction $u$ vanishes on $D$ and $D$ is connected, $D$ is contained in the closure of a single nodal domain. There is thus another nodal domain whose closure does not intersect $D$ on which $w>0$. For the first mixed eigenvalue, we need to construct only a one-dimensional test space, so this other nodal domain yields the contradiction in this case. 
\end{proof}

\section{Examples}\label{exsection}

Here we explain the three examples mentioned in Section \ref{intro}. 

\subsection{Ideal polygons with many vertices}
Let $\Omega$ denote the interior of the convex hull of $n<\infty$ points in the ideal boundary $\partial_{\infty}\Hbb^2$. We call $\Omega$ an \textit{ideal polygon}, and we refer to the $n$ points in $\partial_{\infty}\Hbb^2$ as the \textit{vertices} of $\Omega$. We claim that ideal polygons have discrete Neumann spectrum in the interval $(0,1/4)$ and, if $n\geq 18$, then $\Omega$ has a Neumann eigenvalue in $(0,1/4)$. Since $\Omega$ is convex, Theorem \ref{neumann14} implies that a second Neumann eigenfunction has no interior critical points. \\
\indent First, note that the bottom of the spectrum of the Dirichlet Laplacian on $\Omega$ is greater than $1/4$. Otherwise, we could extend a corresponding eigenfunction to equal $0$ in $\Hbb^2\setminus \Omega$ to create a test function that contradicts (\ref{14eqn}).\\
\indent Let $S$ equal the hyperbolic surface obtained by gluing two copies of $\Omega$ together along their boundary. Then $S$ is a hyperbolic $n$-punctured sphere. This surface has discrete spectrum in the interval $[0,1/4)$ consisting only of eigenvalues. By Theorem 2 \cite{zograf}, $S$ has a Laplace eigenvalue in $(0,1/4)$. The eigenfunctions on $S$ can be decomposed into sums of eigenfunctions that are all even or odd with respect to the symmetry exchanging the two copies of $\Omega$. The even functions restrict to Neumann eigenfunctions on $\Omega$, and the odd functions restrict to Dirichlet eigenfunctions. Since $\Omega$ has no Dirichlet eigenvalues in $(0,1/4)$, it follows that $\Omega$ has a Neumann eigenvalue in this interval. 

\subsection{Large bounded convex domains}\label{largesubsection}
Corollary \ref{largedomains} is a consequence of Theorem \ref{neumann14}, Chavel's extension of the Szeg\"o--Weinberger inequality to hyperbolic planar domains, and the fact that the second Neumann eigenvalue of hyperbolic disks tends to $0$ as the radius tends to infinity.
\begin{proof}[Proof of Corollary \ref{largedomains}]
    By Theorem \ref{neumann14}, it suffices to show that there exists $A>0$ such that for all bounded convex domains of area at least $A$, the second Neumann eigenvalue is at most $1/4$. To our knowledge, Chavel was the first to point out in \cite{chavel2} that Weinberger's proof \cite{weinberger} of the Szeg\"o-Weinberger inequality can be extended in a simple way to prove that among simply connected hyperbolic planar domains of the same area, the disk maximizes the second Neumann eigenvalue. Thus, it suffices to find a sufficient condition for a hyperbolic disk to have second Neumann eigenvalue at most $1/4$. Theorem 2 of \cite{langlaug2} implies that the second Neumann eigenvalue of a hyperbolic disk is strictly decreasing to $0$ in the radius. As the radius tends to $0$, the eigenvalue approaches positive infinity, so there exists a unique $r_0\in (0,\infty)$ for which the hyperbolic disk of radius $r_0$ has second Neumann eigenvalue equal to $1/4$. Thus it suffices to take $A=2\pi(\cosh(r_0)-1)$ equal to the area of this disk. 
\end{proof}  

The Neumann eigenvalues of a hyperbolic disk can be computed in terms of zeros of derivatives of associated Legendre functions via separation of variables. In particular, the radius $r_0$ in the proof of Corollary \ref{largedomains} above is the first positive zero of the derivative of the function \[r\mapsto P_{-1/2}^1(\cosh(r)).\] The numerical value of $r_0$ is approximately $2.52851$. Thus, \[A(r_0)=2\pi\big(\cosh(r_0)-1\big)\leq 33.35,\] establishing the claim in Remark \ref{numremk}.

\subsection{First mixed eigenfunctions with small Dirichlet region}
By a test function argument very similar to the one given in \cite{me3}, we show that, on a bounded convex domain $\Omega\subseteq \Hbb^2$, the first mixed eigenvalue tends to zero as the diameter of the Dirichlet region shrinks to zero. As a result, Theorem \ref{mixed14} implies that for $\epsilon$ sufficiently small, for any connected set $D\subseteq \partial\Omega$ of diameter at most $\epsilon$, a first mixed eigenfunction has no interior critical points.\\
\indent Let $\epsilon\in (0,1)$. Using polar coordinates in the Poincar\'e disk model of $\Hbb^2$, define a function \[u_{\epsilon}(r,\theta):=\begin{cases}
    0\;\;&\text{if}\;\;r<\epsilon\\[5pt]
    1-\ln(r/\sqrt{\epsilon})/\ln(\sqrt{\epsilon})&\text{if}\;\;\epsilon\leq r\leq \sqrt{\epsilon}\\[5pt]
    1&\text{if}\;\;\sqrt{\epsilon}\leq r.
\end{cases}\]
Then we have
\begin{equation}\label{limitsofue}
    \int_{\Omega}|\nabla u_{\epsilon}|^2\to 0\;\;\text{and}\;\;\int_{\Omega}|u_{\epsilon}|^2\to |\Omega|\;\;\text{as}\;\;\epsilon\to0.
\end{equation}
If $D\subseteq \partial\Omega$ has sufficiently small diameter, we may apply an isometry so that $D$ is contained in the set $\{0\leq r\leq \epsilon\}$. By (\ref{limitsofue}), for $\epsilon$ sufficiently small, the first mixed eigenvalue is less than $1/4$. Since $\partial\Omega$ is compact, we may find a uniform $\epsilon$ that gives this estimate for any connected subset $D$ that has diameter less than $\epsilon$.

\printbibliography

\end{document}